\documentclass[reqno,12pt]{amsart}
\usepackage[colorlinks=true, linkcolor=blue, citecolor=blue]{hyperref}
    
\usepackage{amssymb}
\usepackage{amsmath, graphicx, rotating}
\usepackage{color}
\usepackage{soul}
\usepackage[dvipsnames]{xcolor}
               
\usepackage{ifthen}      
\usepackage{xkeyval}
\usepackage{todonotes}     
\usepackage[T1]{fontenc}
\usepackage{lmodern}
\usepackage[english]{babel}

\usepackage{ upgreek }
\usepackage{stmaryrd}
\SetSymbolFont{stmry}{bold}{U}{stmry}{m}{n}
\usepackage{amsthm}
\usepackage{float}

\usepackage{ bbm }
\usepackage{ stmaryrd }
\usepackage{ mathrsfs }
\usepackage{ frcursive }
\usepackage{ comment }

\usepackage{pgf, tikz}
\usetikzlibrary{shapes}
\usepackage{varioref}
\usepackage{enumitem}

\usepackage{mathtools}  

\usepackage{dsfont}

\definecolor{rouge}{rgb}{0.7,0.00,0.00}
\definecolor{vert}{rgb}{0.00,0.5,0.00}
\definecolor{bleu}{rgb}{0.00,0.00,0.8}

\usepackage[margin=1.2in]{geometry}

\newtheorem{theorem}{Theorem}[section]
\newtheorem*{theorem*}{Theorem}
\newtheorem{lemma}[theorem]{Lemma}

\labelformat{hypothesis}{\textbf{M\kern-0.1mm#1}}

\newtheorem{condition}{Condition}
\newtheorem{conditionA}{A\kern-0.1mm}
\labelformat{conditionA}{\textbf{A\kern-0.1mm#1}}

\renewcommand\dots{\hbox to 1em{.\hss.\hss.}}

\theoremstyle{definition}

\numberwithin{equation}{section}

\def\bb#1{\mathbb{#1}}

\def\scr#1{\mathscr{#1}}

\def\geq{\geqslant}
\def\leq{\leqslant}

\newcommand{\ee}{\varepsilon}

\DeclareMathOperator{\supp}{supp}

\DeclarePairedDelimiter\floor{\lfloor}{\rfloor}

\begin{document}

\title[Edgeworth expansion for coefficients]
{Edgeworth expansion for the coefficients of random walks on the general linear group}


\author{Hui Xiao}
\author{Ion Grama}
\author{Quansheng Liu}

\curraddr[Xiao, H.]{ Universit\'{e} de Bretagne-Sud, LMBA UMR CNRS 6205, Vannes, France}
\email{hui.xiao@univ-ubs.fr}
\curraddr[Grama, I.]{ Universit\'{e} de Bretagne-Sud, LMBA UMR CNRS 6205, Vannes, France}
\email{ion.grama@univ-ubs.fr}
\curraddr[Liu, Q.]{ Universit\'{e} de Bretagne-Sud, LMBA UMR CNRS 6205, Vannes, France}
\email{quansheng.liu@univ-ubs.fr}


\begin{abstract}
Let $(g_n)_{n\geq 1}$ be a sequence of independent and identically distributed 
random elements with law $\mu$ on the general linear group $\textup{GL}(V)$, where $V=\bb R^d$. 
Consider the random walk $G_n : = g_n \ldots g_1$, $n \geq 1$.
Under suitable conditions on $\mu$, we establish the first-order Edgeworth expansion
for the coefficients $\langle f, G_n v \rangle$ with $v \in V$ and $f \in V^*$, 
in which a new additional term appears compared to the case of vector norm $\|G_n v\|$. 
\end{abstract}

\date{\today}
\subjclass[2010]{Primary 60F05, 60F15, 60F10; Secondary 37A30, 60B20}
\keywords{Random walks on groups; coefficients;  central limit theorem;  Edgeworth expansion; Berry-Esseen bound}

\maketitle



\section{Introduction} 
Since the pioneering work of Furstenberg and Kesten \cite{FK60}, 
the study of random walks on linear groups 
has attracted a great deal of attention, 
see for instance the work of Le Page \cite{LeP82}, 
Guivarc'h and Raugi \cite{GR85},  Bougerol and Lacroix \cite{BL85}, Goldsheid and Margulis \cite{GM89},
Benoist and Quint \cite{BQ16b}, and the references therein. 
Of particular interest is the study of asymptotic properties of the random walk
$G_n : = g_n \ldots g_1$, $n \geq 1$, 
where $(g_n)_{n \geq 1}$ is a sequence of independent and identically distributed
random elements with law $\mu$ on the general linear group $\textup{GL}(V)$ with $V = \bb R^d$. 
One natural and important way to describe the random walk $(G_n)_{n\geq 1}$ is to  
investigate the growth rate of the coefficients $\langle f, G_n v \rangle$, 
where $v \in V$, $f \in V^*$ and $\langle \cdot, \cdot \rangle$ is the duality bracket: $\langle f, v \rangle = f(v)$. 
Bellman \cite{Bel54} conjectured that the classical central limit theorem 
should hold true for $\langle f, G_n v \rangle$ 
in the case when $g_n$ are positive matrices. 
This conjecture was proved by Furstenberg and Kesten \cite{FK60},
who established the strong law of large numbers and central limit theorem 
under the condition that the matrices $g_n$ are strictly positive and that all the coefficients of $g_n$ are comparable. 
For further developments we refer to Kingman \cite{Kin73}, 
Cohn, Nerman and Peligrad \cite{CNP93}, Hennion \cite{Hen97}.

As noticed  by Furstenberg \cite{Fur63}, 
the analysis developed in \cite{FK60} for positive matrices breaks down for invertible matrices.  
It turns out that the situation of invertible matrices is much more complicated and delicate. 
Guivarc'h and Raugi \cite{GR85} established the strong law of large numbers for the coefficients of products 
of invertible matrices
under an exponential moment condition: 
for any $v \in V \setminus \{0\}$ and $f \in V^* \setminus \{0\}$, 
\begin{align}\label{Ch7_SLLN_Entry0a}
\lim_{n\to\infty} \frac{1}{n} \log | \langle f, G_n v \rangle | = \lambda \quad  \mbox{a.s.}, 
\end{align}
where $\lambda \in \bb R$ is a constant independent of $f$ and $v$, called the first Lyapunov exponent of $\mu$.  
It is worth mentioning that the result \eqref{Ch7_SLLN_Entry0a} does not follow from
the classical subadditive ergodic theorem of Kingman \cite{Kin73}, nor from its recent version by Gou\"ezel and Karlsson \cite{GK20}.  
The central limit theorem for the coefficients has also been established in \cite{GR85}  
under the exponential moment condition: 
if $\int_{ \textup{GL}(V) }  N(g)^{\ee} \mu(dg) < \infty$ with $N(g) = \max \{ \|g\|, \| g^{-1} \| \}$ for some $\ee > 0$, 
then for any $t \in \bb R$, 
\begin{align}\label{Ch7_CLT_Entry0a}
\lim_{n \to \infty} \bb{P} \left( \frac{ \log | \langle f, G_n v \rangle | 
 - n \lambda}{ \sigma \sqrt{n} } \leq t  \right)  =  \Phi(t), 
\end{align}
where $\Phi$ is the standard normal distribution function on $\bb R$
and $\sigma^2 > 0$ is the asymptotic variance of $ \frac{1}{\sqrt n} \log | \langle f, G_n v \rangle |.$  
Recently, using the martingale approximation method, Benoist and Quint \cite{BQ16} have improved \eqref{Ch7_CLT_Entry0a}
by relaxing the exponential moment condition to the optimal second moment 
$\int_{ \textup{GL}(V) }  (\log N(g))^{2} \mu(dg) < \infty$.   

An important and interesting problem is the estimation of the rate of convergence in \eqref{Ch7_CLT_Entry0a}.
Very recently, under the exponential moment condition,  Cuny, Dedecker, Merlev\`ede and Peligrad \cite{CDMP21b} 
established a rate of convergence of order $\log n/\sqrt{n}$.
Dinh, Kaufmann and Wu \cite{DKW21, DKW21b} improved this result by giving the optimal rate $1/\sqrt{n}$
under the same exponential moment assumption:
there exists a constant $c > 0$ such that for all $n \geq 1$, $t \in \bb R$, 
$v \in V$ and $f \in V^*$  with $\|v\| = \|f\| =1$, 
\begin{align} \label{BerryEsseen_Coeffaa-Intro}
\left|  \bb{P} \left(  \frac{\log |\langle f, G_n v \rangle| - n \lambda }{ \sigma \sqrt{n} } \leq t    \right)
-  \Phi(t)  \right|  \leq  \frac{c }{\sqrt{n}}. 
\end{align}

The objective of this paper is to further elaborate on the central limit theorem \eqref{Ch7_CLT_Entry0a}
by establishing the first-order Edgeworth expansion for the coefficients  
under the exponential moment condition. We prove that 
as $n\to \infty$,  
uniformly in $t \in \bb R$, $x=\bb R v \in \bb P(V)$ and $ y = \bb R f \in \bb P(V^*)$ with $\|v\| = \|f\| =1$,    
\begin{align}\label{Edgeworth-Coeff-Intro}
& \bb{P} \left(  \frac{\log |\langle f, G_n v \rangle| - n \lambda }{ \sigma \sqrt{n} } \leq t    \right)   \notag\\
 &   =   \Phi(t) + \frac{\Lambda'''(0)}{ 6 \sigma^3 \sqrt{n}} (1-t^2) \phi(t)
  - \frac{ b_{1}(x) + d_{1}(y) }{ \sigma \sqrt{n} } \phi(t)    +  o \Big( \frac{ 1 }{\sqrt{n}} \Big),   
\end{align}
where $\phi$ denotes the standard normal density,  $\Lambda'''(0)$, $b_{1}(x)$, $d_{1}(y)$ 
are  defined in Section \ref{Sec-main-result}. 
Notice that the asymptotic bias terms $b_{1}(x)$ and $d_{1}(y)$ are new compared with the classical Edgeworth expansion
for sums of independent real random variables \cite{Pet75}; 
 $d_{1}(y)$ is also new compared with the Edgeworth expansion for the vector norm $\|G_n v \|$ \cite{XGL19b}.
In fact, we will establish a stronger result, that is, the first-order Edgeworth expansion 
for the couple $(\varphi(G_n \!\cdot\! x), \log |\langle f, G_n v \rangle|)$ with a target function $\varphi$ on $\bb P(V)$, 
cf.\ Theorem \ref{Thm-Edge-Expan-Coeff001}.
Moreover, we prove a similar result under the changed measure, which can be useful for studying 
moderate deviations with explicit rates of convergence. 
Clearly, the expansion \eqref{Edgeworth-Coeff-Intro} implies the Berry-Esseen bound  \eqref{BerryEsseen_Coeffaa-Intro}.

The proof of the Edgeworth expansion for the coefficient $\langle f, G_n v \rangle$  
turns out to be much more complicated than that for the norm cocycle $\sigma (G_n, x)= \log \frac{\| G_n v \|}{\|v\|}$, 
$x = \bb R v \in \bb P(V)$ recently established in \cite{XGL19b}. 
One of the difficulties is that 
$\log | \langle f, G_n v \rangle |$ is not a cocycle and cannot be studied with the same approach as 
$\sigma (G_n, x)$.
Our starting point is the following decomposition
which relates the coefficient to the norm cocycle: 
for any $x = \bb R v \in \bb P(V)$ and $y = \bb R f \in \bb P(V^*)$ with $\|f\|=1$, 
\begin{align}\label{Ch7_Intro_Decom0a}
\log |\langle f, G_n v \rangle| = \sigma (G_n, x) +  \log \delta(G_n \!\cdot\! x, y),  
\end{align}
where  $\delta(x, y) = \frac{|\langle f, v \rangle|}{\|f\| \|v\|}$. 
For the proof of the Edgeworth expansion \eqref{Edgeworth-Coeff-Intro}, 
we first use a partition $(\chi_{n,k}^y)_{k \geq 1}$ of the unity  to discretize the component $\log \delta(G_n \!\cdot\! x, y)$ in \eqref{Ch7_Intro_Decom0a}.
This allows us to reduce the study of the coefficient to that of the couple formed by norm cocycle $\sigma(G_n, x)$ 
and the target function $\chi_{n,k}^y(G_n \!\cdot\! x)$. 
It turns out that the Edgeworth expansion 
for the couple $ ( \chi_{n,k}^y  (G_n \!\cdot\! x), \sigma(G_n, x))$
established recently in \cite{XGL19b}
is not appropriate for our proof because the reminder terms therein are not precise enough. 
We need to track the dependence of the remainder term on the H\"older norm of the function $\varphi = \chi_{n,k}^y$, 
see Theorem \ref{Thm-Edge-Expan}. 
In contrast to the previous work \cite{DKW21}, 
the partition of the unity that we use should become finer and finer as $n \to \infty$, in order to recover the term $d_{1}(y)$, 
see Lemma \ref{new bound for delta020}.  
Finally, another delicate point is to 
patch up the expansions for couples  $(\chi_{n,k}^y (G_n \!\cdot\! x), \sigma(G_n, x))$
by means of the H\"older regularity of the invariant measure $\nu$
and the linearity in $\varphi$ of the asymptotic bias term $b_{\varphi}(x)$.  
\section{Main results}\label{Sec-main-result}

For any integer $d \geq 1$, denote by $V = \bb R^d$ the $d$-dimensional Euclidean space.
We fix a basis $e_1, \ldots, e_d$ of $V$ and the associated norm on $V$ is defined by $\|v\|^2 = \sum_{i=1}^d |v_i|^2$ for $v = \sum_{i=1}^d v_i e_i \in V$. 
Let $V^*$ be the dual vector space of $V$ and its dual basis is denoted by $e_1^*, \ldots, e_d^*$ 
so that $e_i^*(e_j)= 1$ if $i=j$ and  $e_i^*(e_j)= 0$ if $i\neq j$. 
Let $\wedge^2 V$ be the exterior product of $V$
and we use the same symbol $\| \cdot \|$ for the norms induced on $\wedge^2 V$ and $V^*$.  
We equip $\bb P (V)$ with the angular distance 
\begin{align}\label{Angular-distance}
d(x, x') = \frac{\| v \wedge v' \|}{ \|v\| \|v'\| }  \quad \mbox{for} \  x= \bb R v \in \bb P(V), \  x' = \bb R v' \in \bb P(V). 
\end{align}
We use the symbol $\langle \cdot, \cdot \rangle$ 
to denote the dual bracket defined by $\langle f, v \rangle = f(v)$ for any $v \in V$ and $f \in V^*$. 
Set 
\begin{align*}
\delta(x,y) = \frac{| \langle f, v \rangle |}{\|f\| \|v\| }  \quad \mbox{for} \  x= \bb R v \in \bb P(V),  \  y = \bb R f \in \bb P(V^*).
\end{align*}
Denote by $\mathscr{C}(\bb{P}(V) )$ the space of complex-valued continuous functions on $\bb{P}(V)$, 
equipped with the norm $\|\varphi\|_{\infty}: =  \sup_{x\in \bb{P}(V) } |\varphi(x)|$ for $\varphi \in \mathscr{C}(\bb{P}(V) )$.
Let $\gamma>0$ be a constant and set
\begin{align*}
\|\varphi\|_{\gamma}: =  \|\varphi\|_{\infty} + [\varphi]_{\gamma}, \quad  \mbox{where}  \  
[\varphi]_{\gamma} =  \sup_{x, x' \in \bb{P}(V): x \neq x'} \frac{|\varphi(x)-\varphi(x')|}{ d(x, x')^{\gamma} }.  
\end{align*}
Consider the Banach space 
\begin{align*}
\mathscr{B}_{\gamma}:= \left\{ \varphi\in \mathscr{C}(\bb P(V)): \|\varphi\|_{\gamma}< \infty \right\},   
\end{align*}
which consists of complex-valued $\gamma$-H\"older continuous functions on $\bb P(V)$. 
Denote by $\mathscr{L(B_{\gamma},B_{\gamma})}$ 
the set of all bounded linear operators from $\mathscr{B}_{\gamma}$ to $\mathscr{B}_{\gamma}$
equipped with the operator norm
$\left\| \cdot \right\|_{\mathscr{B}_{\gamma} \to \mathscr{B}_{\gamma}}$. 
The topological dual of $\mathscr B_\gamma$ endowed with the induced norm 
is denoted by $\mathscr B'_\gamma$.
Let $\mathscr B_\gamma^*$ be the Banach space of $\gamma$-H\"older continuous functions
on $\mathbb P(V^*)$ endowed with the norm
\begin{align*} 
\| \varphi \|_{\mathscr B_\gamma^*} = 
 \sup_{y\in \mathbb P(V^*) } |\varphi(y)| 
  + \sup_{  y,y'\in \bb P(V^*): \, y \neq y' } \frac{ |\varphi(y)-\varphi(y')| }{ d(y,y')^{\gamma} }, 
\end{align*}
where $d(y,y') = \frac{\| f \wedge f' \|}{ \|f\| \|f'\| }$ for $y= \bb R f \in \bb P(V^*)$ and $y' = \bb R f' \in \bb P(V^*)$.

Let $\textup{GL}(V)$ be the general linear group of the vector space $V$.
The action of $g \in \textup{GL}(V)$ on a vector $v \in V$ is denoted by $gv$, 
and the action of $g \in \textup{GL}(V)$ on a projective line $x = \bb R v \in \bb P(V)$ is denoted by $g \cdot x = \bb R gv$. 
For any $g \in \textup{GL}(V)$, let $\| g \| = \sup_{v \in V \setminus \{0\} } \frac{\| g v \|}{\|v\|}$
and denote $N(g) = \max \{ \|g\|, \| g^{-1} \| \}$. 
Let $\mu$ be a Borel probability measure on $\textup{GL}(V)$. 

We shall use the following exponential moment condition. 

\begin{conditionA}
\label{Ch7Condi-Moment} 
There exists a constant $\ee >0$ such that $\int_{ \textup{GL}(V) } N(g)^{\ee} \mu(dg) < \infty$. 
\end{conditionA}

Let $\Gamma_{\mu}$ be the smallest closed subsemigroup
generated by the support of the measure $\mu$. 
An endomorphism $g$ of $V$ is said to be proximal 
if it has an eigenvalue $\lambda$ with multiplicity one and all other eigenvalues of $g$ have modulus strcitly less than $|\lambda|$.
We shall need the following strong irreducibility and proximality condition. 

\begin{conditionA}\label{Ch7Condi-IP}
{\rm (i)(Strong irreducibility)} 
No finite union of proper subspaces of $V$ is $\Gamma_{\mu}$-invariant.

{\rm (ii)(Proximality)}
$\Gamma_{\mu}$ contains a proximal endomorphism. 
\end{conditionA}

Define the norm cocycle $\sigma: \textup{GL}(V) \times \bb P(V) \to \bb R$ as follows: 
\begin{align*}
\sigma (g, x) = \log \frac{\|gv\|}{\|v\|} \quad   \mbox{for any} \  g \in \textup{GL}(V)  \  \mbox{and}  \   x = \bb R v \in \bb P(V). 
\end{align*}
Recall that the first Lyapunov exponent $\lambda$ is defined by \eqref{Ch7_SLLN_Entry0a}. 
By \cite[Proposition 3.15]{XGL19b}, under \ref{Ch7Condi-Moment} and \ref{Ch7Condi-IP},
 the following limit exists and is independent of $x \in \bb P(V)$:   
\begin{align}\label{Def-sigma}
\sigma^2: = \lim_{n \to \infty} \frac{1}{n} \bb E \left[ (\sigma (G_n, x) - n \lambda)^2 \right] \in (0, \infty).  
\end{align}
For any $s \in (-s_0, s_0)$ with $s_0 >0$ small enough, 
we define the transfer operator $P_s$ as follows: for any bounded measurable function $\varphi$ on $\bb P(V)$, 
\begin{align}\label{Def_Ps001}
P_s \varphi(x) = \int_{ \textup{GL}(V) } e^{s \sigma(g, x)} \varphi(g \!\cdot\! x) \mu(dg),  
\quad  x \in \bb P(V). 
\end{align}
It will be shown in Lemma \ref{Ch7transfer operator} that 
there exists a constant $s_0 >0$ such that for any $s \in (-s_0, s_0)$,
the operator $P_s \in \mathscr{L(B_{\gamma},B_{\gamma})}$ 
has 
a unique dominant eigenvalue $\kappa(s)$ 
with $\kappa(0) = 1$ and the mapping $s \mapsto \kappa(s)$ being analytic. 
We denote $\Lambda = \log \kappa$.

Under \ref{Ch7Condi-Moment} and \ref{Ch7Condi-IP},
the Markov chain $(G_n \!\cdot\! x)_{n \geq 0}$ has a unique invariant probability measure 
$\nu$ on $\bb P(V)$ such that for any bounded measurable function $\varphi$ on $\bb P(V)$,
\begin{align} \label{Ch7mu station meas}
\int_{\bb P(V)} \int_{\textup{GL}(V)} \varphi(g \!\cdot\! x) \mu(dg) \nu(dx) 
 = \int_{ \bb P(V) } \varphi(x) \nu(dx)
= : \nu(\varphi).
 \end{align}

For any $\varphi \in \mathscr{B}_{\gamma}$, define the functions
\begin{align} \label{drift-b001}
b_{\varphi}(x): = \lim_{n \to \infty}
   \mathbb{E} \big[ ( \sigma(G_n, x) - n \lambda ) \varphi(G_n \!\cdot\! x) \big], 
\quad   x \in \bb P(V) 
\end{align}
and
\begin{align} \label{drift-d001}
d_{\varphi}(y) :=  \int_{\bb P(V)} \varphi(x) \log \delta(x, y)\nu(dx), \quad  y \in \bb P(V^*). 
\end{align}
It will be shown in Lemmas \ref{Lem-Bs} and \ref{Lem-ds} that
both functions $b_{\varphi}$ and $d_{\varphi}$ are well-defined and $\gamma$-H\"older continuous. 
Denote $\phi(u) = \frac{1}{\sqrt{2 \pi}} e^{- u^2/2}$,  $u \in \bb R$. 
Let
 $\Phi(t) =  \int_{- \infty}^t  \phi(u) du$, $t \in \bb R$
be the standard normal distribution function. 

In many applications it is of primary interest to give an estimation of the rate of convergence in the Gaussian approximation \eqref{Ch7_CLT_Entry0a}.  
In this direction we establish the following first-order Edgeworth expansion for the coefficients $\langle f, G_n v \rangle$. 
 
\begin{theorem}\label{Thm-Edge-Expan-Coeff001}
Assume \ref{Ch7Condi-Moment} and \ref{Ch7Condi-IP}. 
Then, there exists a constant $\gamma >0$ such that for any $\ee > 0$, uniformly in 
 $t \in \bb R$, $x=\bb R v \in \bb P(V)$, $ y = \bb R f \in \bb P(V^*)$ with $\|v\| = \|f\| =1$, and
  $\varphi \in \mathscr{B}_{\gamma}$,  as $n\to \infty$,
\begin{align}\label{EdgeworthExpan}
& \mathbb{E}
   \Big[  \varphi(G_n \!\cdot\! x) \mathds{1}_{ \big\{ \frac{\log |\langle f, G_n v \rangle| - n \lambda  }{\sigma \sqrt{n}} \leq t \big\} } \Big]
    = \nu(\varphi) \Big[  \Phi(t) + \frac{\Lambda'''(0)}{ 6 \sigma^3 \sqrt{n}} (1-t^2) \phi(t) \Big]
     \notag\\
& \qquad\qquad\qquad 
 - \frac{ b_{\varphi}(x) + d_{\varphi}(y) }{ \sigma \sqrt{n} } \phi(t) 
  +  \nu(\varphi)  o \Big( \frac{ 1 }{\sqrt{n}} \Big)  +  \lVert \varphi \rVert_{\gamma} O \Big( \frac{ 1 }{n^{1 - \ee} } \Big).     
\end{align}
\end{theorem}

When compared with the standard Edgeworth expansion for sums of independent random variables (cf.\ \cite{Pet75}),
we see that two new terms $b_{\varphi}(x)$ and $d_{\varphi}(y)$ show up, 
which are explained by the presence of an asymptotic bias for this model. 
We should also note that the Edgeworth expansion \eqref{EdgeworthExpan} for the coefficients is different from
 that for the norm cocycle $\sigma (G_n, x)$ obtained in \cite{XGL19b}, 
 namely, by the presence of the term $d_{\varphi}(y)$. 
 The difficulty in proving this precise expansion for coefficient $\langle f, G_n v \rangle$ 
 consists in obtaining the exact expression of this new asymptotic bias term $d_{\varphi}(y)$.

As a consequence of Theorem \ref{Thm-Edge-Expan-Coeff001} one can get the Berry-Esseen bound \eqref{BerryEsseen_Coeffaa-Intro}
with the optimal convergence rate, under the exponential moment condition.
It is an open problem how to relax the exponential moment condition \ref{Ch7Condi-Moment}
for the Edgeworth expansion and for the Berry-Esseen bound.
Solving it seems very challenging.  
Even for the easier case of the norm cocycle, the 
Berry-Esseen bound $O(n^{-1/2})$ is not known under the optimal third moment condition;
it is only known under the fourth moment condition, see \cite{CDMP21}.
For positive matrices, the Edgeworth expansion \eqref{EdgeworthExpan}
and the Berry-Esseen bound \eqref{BerryEsseen_Coeffaa-Intro} have been recently obtained
using a different approach in a forthcoming paper \cite{XGL21b}
under optimal moment conditions.  



Finally we would like to mention that all the results of the paper remain valid when $V$ is $\bb C^d$ or $\bb K^d$, where 
$\bb K$ is any local field. 

\section{Proof of the Edgeworth expansion}

\subsection{Preliminary results}\label{subsec-Pz}

For any $z \in \bb{C}$, 
we define the complex transfer operator $P_z$ as follows: 
for any bounded measurable function $\varphi$ on $\bb P(V)$, 
\begin{align}\label{Def_Pz_Ch7}
P_z \varphi(x) = \int_{\textup{GL}(V)} e^{z \sigma(g, x)} \varphi(g \!\cdot\! x) \mu(dg),  
\quad  x \in \bb P(V). 
\end{align}
Throughout this paper let $B_{s_0}(0): = \{ z \in \bb{C}: |z| < s_0 \}$
be the open disc with center $0$ and radius $s_0 >0$ in the complex plane $\bb C$. 
The following result  
shows that the operator $P_z$ has spectral gap properties when $z \in B_{s_0}(0)$;  
we refer to \cite{LeP82, HH01, GL16, BQ16b, XGL19b} for the proof 
based on the perturbation theory of linear operators. 
Recall that $\mathscr{B}_{\gamma}'$ is the topological dual space of the Banach space $\mathscr{B}_{\gamma}$, 
and that $\mathscr{L(B_{\gamma},B_{\gamma})}$ 
is the set of all bounded linear operators from $\mathscr{B}_{\gamma}$ to $\mathscr{B}_{\gamma}$
equipped with the operator norm
$\left\| \cdot \right\|_{\mathscr{B}_{\gamma} \to \mathscr{B}_{\gamma}}$.

\begin{lemma}[\cite{BQ16b, XGL19b}]  \label{Ch7transfer operator}
Assume \ref{Ch7Condi-Moment} and \ref{Ch7Condi-IP}.  
Then, there exists a constant $s_0 >0$ such that for any $z \in B_{s_0}(0)$ and $n \geq 1$, 
\begin{align}\label{Ch7Pzn-decom}
P_z^n = \kappa^n(z) \nu_z \otimes r_z + L_z^n, 
\end{align}
where 
\begin{align*}
z \mapsto \kappa(z) \in \bb{C}, \quad   z \mapsto r_z \in \mathscr{B}_{\gamma} , 
\quad   z \mapsto \nu_z \in \mathscr{B}_{\gamma}' , 
\quad   z \mapsto  L_z \in \mathscr{L(B_{\gamma},B_{\gamma})}
\end{align*}
are analytic mappings which satisfy, for any $z \in B_{s_0}(0)$, 

\begin{itemize}
\item[{\rm(a)}]
    the operator $M_z: = \nu_z \otimes r_z$ is a rank one projection on $\mathscr{B}_{\gamma}$,
    i.e. $M_z \varphi = \nu_z(\varphi) r_z$ for any $\varphi \in \mathscr{B}_{\gamma}$; 

\item[{\rm(b)}]
 $M_z L_z = L_z M_z =0$,  $P_z r_z = \kappa(z) r_z$ with $\nu(r_z) = 1$, and  $\nu_z P_z = \kappa(z) \nu_z$;

\item[{\rm(c)}]
    $\kappa(0) = 1$, $r_0 = 1$, $\nu_0 = \nu$ with $\nu$ defined by \eqref{Ch7mu station meas}, and 
    $\kappa(z)$ and $r_z$ are strictly positive for real-valued $z \in (-s_0, s_0)$.    
\end{itemize}
\end{lemma}

Using Lemma \ref{Ch7transfer operator}, 
a change of measure can be performed below.
For any $s \in (-s_0, s_0)$ with $s_0>0$ sufficiently small, 
 any $x \in \bb P(V)$ and $g \in \textup{GL}(V)$, denote
\begin{align*}
q_n^s(x, g) = \frac{ e^{s \sigma(g, x) } }{ \kappa^{n}(s) } \frac{ r_s(g \!\cdot\! x) }{ r_s(x) },
\quad  n \geq 1. 
\end{align*}
Since the eigenvalue $\kappa(s)$ and the eigenfunction $r_s$ are strictly positive for $s \in (-s_0, s_0)$, 
using $P_s r_s = \kappa(s) r_s$ we get that 
\begin{align*}
\bb Q_{s,n}^x (dg_1, \ldots, dg_n) = q_n^s(x, G_n) \mu(dg_1) \ldots \mu(dg_n),  \quad  n \geq 1, 
\end{align*}
are probability measures and 
form a projective system on $\textup{GL}(V)^{\bb{N}}$. 
By the Kolmogorov extension theorem, 
there exists a unique probability measure  $\bb Q_s^x$ on $\textup{GL}(V)^{\bb{N}}$ with marginals $\bb Q_{s,n}^x$. 
We write $\bb{E}_{\bb Q_s^x}$ for the corresponding expectation 
and the change of measure formula holds: 
for any $s \in (-s_0, s_0)$, $x \in \bb P(V)$, 
$n\geq 1$ and bounded measurable function $h$ on $(\bb P(V) \times \bb R)^{n}$,  
\begin{align}\label{Ch7basic equ1}
&  \frac{1}{ \kappa^{n}(s) r_{s}(x) } \bb{E}  \Big[   
r_{s}(G_n \!\cdot\! x) e^{s \sigma(G_n, x) } h \Big( G_1 \!\cdot\! x, \sigma(G_1, x), \dots, G_n \!\cdot\! x, \sigma(G_n, x) \Big)
  \Big]   \nonumber\\
&  =   \bb{E}_{\bb{Q}_{s}^{x}} \Big[ h \Big( G_1 \!\cdot\! x, \sigma(G_1, x), \dots, G_n \!\cdot\! x, \sigma(G_n, x) \Big) \Big].
\end{align}
Under the changed measure $\bb Q_s^x$, the process $(G_n \!\cdot\! x)_{n \geq 0}$ is a Markov chain 
with the transition operator $Q_s$ given as follows: for any $\varphi \in \mathscr{C}(\bb P(V))$, 
\begin{align*}
Q_{s}\varphi(x) = \frac{1}{\kappa(s)r_{s}(x)}P_s(\varphi r_{s})(x),  \quad   x \in \bb P(V).
\end{align*}
Under \ref{Ch7Condi-Moment} and \ref{Ch7Condi-IP}, 
it was shown in \cite{XGL19b} that 
the Markov operator $Q_s$ has a unique invariant probability measure $\pi_s$ given by 
\begin{align}\label{ExpCon-Qs}
\pi_s(\varphi) = \frac{ \nu_s( \varphi r_s) }{ \nu_s(r_s) }  \quad  \mbox{for any } \varphi \in \mathscr{C}(\bb P(V)).
\end{align}
By \cite[Proposition 3.13]{XGL19b}, the following strong law of large numbers 
for the norm cocycle under the changed measure $\bb Q_s^x$ holds: under \ref{Ch7Condi-Moment} and \ref{Ch7Condi-IP}, 
for any $s\in (-s_0, s_0)$ and $x \in \bb P(V)$, 
\begin{align*}
\lim_{n \to \infty} \frac{ \sigma(G_n, x) }{n} = \Lambda'(s),  \quad  \bb Q_s^x\mbox{-a.s.}
\end{align*}
where $\Lambda(s) = \log \kappa(s)$. 

We need the following H\"older regularity of the invariant measure $\pi_s$. 
\begin{lemma}[\cite{GQX20}]  \label{Lem_Regu_pi_s00}
Assume \ref{Ch7Condi-Moment} and \ref{Ch7Condi-IP}. Then there exist constants $s_0 >0$ and $\eta > 0$ such that
\begin{align} \label{Regu_pi_s_00}
\sup_{ s\in (-s_0, s_0) } \sup_{y \in \bb P(V^*) } 
  \int_{\bb P(V) } \frac{1}{ \delta(x, y)^{\eta} } \pi_s(dx) < + \infty. 
\end{align}
\end{lemma}

We also need the following property: 

\begin{lemma}[\cite{GQX20}]  \label{Lem_Regu_pi_s}
Assume \ref{Ch7Condi-Moment} and \ref{Ch7Condi-IP}. 
Then, for any $\ee >0$, there exist constants $s_0 >0$ and $c, C >0$ such that
for all $s \in (-s_0, s_0)$, $n \geq k \geq 1$, $x \in \bb P(V)$ and $y \in \bb P(V^*)$, 
\begin{align}\label{Regu_pi_s}
\bb Q_s^x \Big( \log \delta(G_n \!\cdot\! x, y) \leq -\ee k  \Big) \leq C e^{- ck}. 
\end{align}
\end{lemma} 

Note that \eqref{Regu_pi_s} is stronger than the exponential H\"{o}lder regularity of the invariant measure $\pi_s$
stated in Lemma  \ref{Lem_Regu_pi_s00}.

\subsection{Proof of Theorem \ref{Thm-Edge-Expan-Coeff001}} \label{sec-proof of Edgeworth exp}
In fact we shall prove a more general version of 
Theorem \ref{Thm-Edge-Expan-Coeff001} 
under the changed measure $\mathbb{Q}_{s}^{x}$.
The proof for the case $s =0$ requires the same effort, so we decide to consider the more general setting.   
For any $s \in (-s_0, s_0)$ and $\varphi \in \mathscr{B}_{\gamma}$,   define 
\begin{align}\label{Def-bsvarphi}
b_{s, \varphi}(x): = \lim_{n \to \infty}
   \mathbb{E}_{\mathbb{Q}_{s}^{x}} \big[ ( \sigma(G_n, x) - n \Lambda'(s) ) \varphi(G_n \!\cdot\! x) \big],
\quad   x \in \bb P(V) 
\end{align}
and
\begin{align} \label{drift-d001bis}
d_{s,\varphi}(y)=  \int_{\bb P(V)} \varphi(x) \log \delta(x, y) \pi_s(dx), \quad  y \in \bb P(V^*). 
\end{align}
These functions are well-defined and $\gamma$-H\"older continuous, as shown in Lemmas \ref{Lem-Bs} and \ref{Lem-ds} below.
In particular, we have $b_{0,\varphi} = b_{\varphi}$ and $d_{0,\varphi} = d_{\varphi}$,
 where $b_{\varphi}$ and $d_{\varphi}$ are defined in \eqref{drift-b001} and \eqref{drift-d001}, respectively. 

Our goal of this subsection is to establish the following first-order Edgeworth expansion for the coefficients $\langle f, G_n v \rangle$
 under the changed measure $\mathbb{Q}_{s}^x$. 
 Note that $\sigma_s = \sqrt{\Lambda''(s)}$, which is strictly positive under \ref{Ch7Condi-Moment} and \ref{Ch7Condi-IP}. 

\begin{theorem}\label{Thm-Edge-Expan-Coeff001extended}
Assume \ref{Ch7Condi-Moment} and \ref{Ch7Condi-IP}. 
Then, 
for any $\ee > 0$,
 there exist $\gamma >0$ and $s_0 > 0$ such that uniformly in $s \in (-s_0, s_0)$, 
 $t \in \bb R$, $x=\bb R v \in \bb P(V)$, $ y = \bb R f \in \bb P(V^*)$ with $\|v\| = \|f\| =1$, and
  $\varphi \in \mathscr{B}_{\gamma}$,  as $n\to \infty$,
\begin{align*}
&  \mathbb{E}_{\mathbb{Q}_{s}^x}
   \Big[  \varphi(G_n \!\cdot\! x) \mathds{1}_{ \big\{ \frac{\log |\langle f, G_n v \rangle| - n \Lambda'(s)  }{\sigma_s \sqrt{n}} \leq t \big\} } \Big]
  = \pi_s(\varphi) \Big[  \Phi(t) + \frac{\Lambda'''(s)}{ 6 \sigma_s^3 \sqrt{n}} (1-t^2) \phi(t) \Big]
     \notag\\
& \qquad\qquad\qquad\qquad\qquad - \frac{ b_{s,\varphi}(x) + d_{s,\varphi}(y) }{ \sigma_s \sqrt{n} } \phi(t) 
 +  \pi_s(\varphi)  o \Big( \frac{ 1 }{\sqrt{n}} \Big)  +  \lVert \varphi \rVert_{\gamma} O \Big( \frac{ 1 }{n^{1-\ee} }\Big).     
\end{align*}
\end{theorem}
Theorem \ref{Thm-Edge-Expan-Coeff001} 
follows from Theorem \ref{Thm-Edge-Expan-Coeff001extended} by taking $s=0$.

The remaining part of the paper is devoted to establishing Theorem \ref{Thm-Edge-Expan-Coeff001extended}. 
We begin with some properties of the function $b_{s, \varphi}$ (cf.\ \eqref{Def-bsvarphi}). 
\begin{lemma}[\cite{XGL19b}] \label{Lem-Bs}
Assume \ref{Ch7Condi-Moment} and \ref{Ch7Condi-IP}. 
Then, 
there exist constants $s_0 >0$, $\gamma>0$ and $c>0$ such that $b_{s,\varphi}  \in \mathscr{B}_{\gamma}$ 
and $\| b_{s,\varphi}\|_{\gamma}  \leq  c \| \varphi \| _{\gamma}$ for any $s\in (-s_0, s_0)$.
\end{lemma}

In addition to Lemma \ref{Lem-Bs}, we shall need the following result on the function $d_{s,\varphi}$ defined in \eqref{drift-d001bis}. 
\begin{lemma} \label{Lem-ds}
Assume \ref{Ch7Condi-Moment} and \ref{Ch7Condi-IP}.
Then, there exists $s_0 >0$ such that for any $s\in (-s_0, s_0)$, the function 
$d_{s,\varphi}$ is well-defined. 
Moreover, there exist constants $\gamma>0$ and $c>0$ such that $d_{s,\varphi}  \in \mathscr{B}^*_{\gamma}$ 
and $\| d_{s,\varphi}\|_{\gamma}  \leq  c \| \varphi \|_{\infty}$ for any $s\in (-s_0, s_0)$.
\end{lemma}

\begin{proof}
Without loss of generality, we assume that $\varphi$ is non-negative. 
Since $\log a\leq a$ for any $a\geq 0$ (with the convention that $\log 0 = -\infty$), 
we have that for any $\eta\in (0,1)$, 
\begin{align}\label{Inequlity-log-eta}
- \eta \log \delta(x, y) \leq \delta(x, y)^{-\eta},  
\end{align}
so that
\begin{align*} 
-d_{s,\varphi}(y) 
\leq \frac{\| \varphi \| _{\infty}}{\eta} \int_{\bb P(V)} \frac{1}{ \delta(x, y)^{\eta} } \pi_s(dx). 
\end{align*}
Choosing $\eta$ small enough,  by Lemma \ref{Lem_Regu_pi_s00}, the latter integral is bounded by some constant uniformly in $y\in \bb P(V^*)$
and $s \in (-s_0, s_0)$,
which proves that $d_{s,\varphi}$ is well-defined and $\|d_{s,\varphi}\|_{\infty} \leq c \| \varphi \|_{\infty}$
for some constant $c>0$.

To estimate $[d_{s,\varphi}]_{\gamma}$, we first note that for any 
$y'=\bb R f'\in \bb P(V^*)$, $y''=\bb R f''\in \bb P(V^*)$ and any $\gamma>0$, 
\begin{align*}
\left| \log \delta(x, y')-\log \delta(x, y'') \right| 
&= \left| \log \delta(x, y')-\log \delta(x, y'') \right| 
   \mathds 1_{ \left\{ \left| \frac{\delta(x, y')-\delta(x, y'')}{\delta(x, y'')} \right|^{\gamma} > \frac{1}{2^{\gamma}} \right\} }   \notag\\
 & \quad    + \left| \log \delta(x, y')-\log \delta(x, y'') \right| 
           \mathds 1_{ \big\{ \big|\frac{\delta(x, y')-\delta(x, y'')}{\delta(x, y'')} \big| \leq \frac{1}{2} \big\} }   \notag\\
 & =: I_1 + I_2. 
\end{align*}
For $I_1$, we easily get 
\begin{align*}
I_1 \leq   2^{\gamma}\left( \left| \log \delta(x, y') \right| + \left| \log \delta(x, y'') \right| \right) 
    \left|\frac{\delta(x, y')-\delta(x, y'')}{\delta(x, y'')} \right|^{\gamma}. 
\end{align*}
For $I_2$, since $|\log(1+a)| \leq  2 |a|$ for any $|a| \leq \frac{1}{2}$,  
we deduce that 
\begin{align*}
I_2  & =   \left| \log \delta(x, y')-\log \delta(x, y'') \right|^{1 - \gamma}
\left| \log \left[ 1 + \frac{ \delta(x, y') - \delta(x, y'') }{ \delta(x, y'') } \right]   \right|^{\gamma}
      \mathds 1_{ \big\{ \big|\frac{\delta(x, y')-\delta(x, y'')}{\delta(x, y'')} \big| \leq \frac{1}{2} \big\} }   \notag\\
& \leq 2^{\gamma} \left| \log \delta(x, y')-\log \delta(x, y'') \right|^{1-\gamma} 
   \left|\frac{\delta(x, y')-\delta(x, y'')}{\delta(x, y'')} \right|^{\gamma}.
\end{align*}
Therefore, 
\begin{align*} 
\left| \log \delta(x, y')-\log \delta(x, y'') \right| 
&\leq  2^{\gamma}\left( \left| \log \delta(x, y') \right| + \left| \log \delta(x, y'') \right| \right) 
    \left|\frac{\delta(x, y')-\delta(x, y'')}{\delta(x, y'')} \right|^{\gamma} \\
& \quad +  2^{\gamma} \left| \log \delta(x, y')-\log \delta(x, y'') \right|^{1-\gamma} 
   \left|\frac{\delta(x, y')-\delta(x, y'')}{\delta(x, y'')} \right|^{\gamma}.
\end{align*}
By \eqref{Inequlity-log-eta}, it holds that $- \gamma \log \delta(x, y) \leq \delta(x, y)^{-\gamma}$ for any $\gamma \in (0,1)$. Hence
there exists a constant $c_{\gamma}>0$ such that 
\begin{align*} 
&|\log \delta(x, y')-\log \delta(x, y'')| \\
&\leq c_{\gamma} \left(\delta(x, y')^{-\gamma} \delta(x, y'')^{-\gamma} + \delta(x, y'')^{ -2\gamma} \right) 
   \left|\delta(x, y')-\delta(x, y'')\right|^ {\gamma}\\
& \quad + c_{\gamma} \left(\delta(x, y')^{-\gamma (1-\gamma) } \delta(x, y'')^{-\gamma} + \delta(x, y'')^{-\gamma (1-\gamma)-\gamma} \right) 
   \left|\delta(x, y')-\delta(x, y'')\right|^ {\gamma}\\
&\leq c_{\gamma} \left(\delta(x, y')^{-\gamma} \delta(x, y'')^{-\gamma} + \delta(x, y'')^{-2\gamma} \right) 
    \left|\delta(x, y')-\delta(x, y'')\right|^ {\gamma}\\
&\leq c_{\gamma} \left(\delta(x, y')^{-2\gamma} +  \delta(x, y'')^{-2\gamma}  \right) 
   \left|\delta(x, y')-\delta(x, y'')\right|^ {\gamma}.
\end{align*}
Since $\| \frac{f'}{\| f' \|} - \frac{f''}{\| f \|} \| \leq \sqrt{2} d(y',y'')$ where $d(y',y'')$ 
is the angular distance on $\bb P(V^*)$,
we have
$$ 
\left|\delta(x, y')-\delta(x, y'')\right|= \left|\frac{ \langle f', v \rangle }{\|v\| \|f'\|} - \frac{ \langle f'', v \rangle }{\|v\| \|f''\|}\right| 
\leq \Big\| \frac{f'}{\| f' \|} - \frac{f''}{\| f \|} \Big\| 
\leq \sqrt{2} d(y',y''). 
$$
By the definition of the function $d_{s,\varphi}$, using the above bounds, we obtain
\begin{align*} 
\frac{| d_{s,\varphi}(y')-d_{s,\varphi}(y'')|}{d(y',y'')^{\gamma}} 
\leq c_{\gamma} \| \varphi \|_{\infty}  
\int_{\bb P(V)} \left(\delta(x, y')^{-2\gamma} + \delta(x, y'')^{-2\gamma}  \right) \pi_s(dx).
\end{align*}
By Lemma \ref{Lem_Regu_pi_s00}, 
the last integral is bounded by some constant uniformly in $y',y''\in \bb P(V^*)$ and $s \in (-s_0, s_0)$,
by choosing $\gamma >0$ sufficiently small.  
This, together with the fact that $\|d_{s,\varphi}\|_{\infty} \leq c \| \varphi \|_{\infty}$, 
proves that 
$d_{s,\varphi} \in \scr B_{\gamma}^*$ and $\| d_{s,\varphi}\|_{\gamma}  \leq  c \| \varphi \|_{\infty}$.
\end{proof}

In the proof of Theorem \ref{Thm-Edge-Expan-Coeff001extended}
we shall make use of the following Edgeworth expansion for the couple $(G_n \cdot x, \sigma(G_n, x))$
with a target function $\varphi$ on $G_n \cdot x$, which slightly improves \cite[Theorem 5.3]{XGL19b} 
by giving more accurate reminder terms. 
This improvement will be important for establishing Theorem \ref{Thm-Edge-Expan-Coeff001extended}. 

\begin{theorem}\label{Thm-Edge-Expan}
Assume \ref{Ch7Condi-Moment} and \ref{Ch7Condi-IP}. 
Then, there exist  constants $s_0 >0$ and $\gamma >0$ such that,   
as $n \to \infty$, uniformly in   $s \in (-s_0, s_0)$, 
$x \in \bb P(V)$, $t \in \bb R$ and $\varphi \in \mathscr{B}_{\gamma}$,  
\begin{align*}
  \mathbb{E}_{\mathbb{Q}_s^x}
   \Big[  \varphi(G_n \cdot x) \mathds{1}_{ \big\{ \frac{\sigma(G_n, x) - n \Lambda'(s) }{\sigma_s \sqrt{n}} \leq t \big\} } \Big]
 &  =  \pi_s(\varphi) \Big[  \Phi(t) + \frac{\Lambda'''(s)}{ 6 \sigma_s^3 \sqrt{n}} (1-t^2) \phi(t) \Big]
    -  \frac{ b_{s,\varphi}(x) }{ \sigma_s \sqrt{n} } \phi(t)     \notag\\
&  \quad  +  \pi_s(\varphi)  o \Big( \frac{ 1 }{\sqrt{n}} \Big)  +  \lVert \varphi \rVert_{\gamma} O \Big( \frac{ 1 }{n} \Big).     
\end{align*}
\end{theorem}

\begin{proof}
For any $x \in \bb P(V)$, define
\begin{align*}
F(t) & = \mathbb{E}_{\mathbb{Q}_s^x}
\Big[  \varphi(G_n \!\cdot\! x) \mathds{1}_{ \big\{ \frac{\sigma(G_n, x) - n \Lambda'(s) }{\sigma_s \sqrt{n}} \leq t \big\} } \Big]
  +  \frac{ b_{s,\varphi}(x) }{ \sigma_s \sqrt{n} } \phi(t),  
  \quad t \in \mathbb{R},   \notag\\
H(t) & =   \mathbb{E}_{\mathbb{Q}_s^x} [ \varphi(G_n \!\cdot\! x) ]
 \Big[ \Phi(t) + \frac{\Lambda'''(s)}{ 6 \sigma_s^3 \sqrt{n}} (1-t^2) \phi(t) \Big],  \quad  t\in \mathbb{R}.
\end{align*}
Since $F(-\infty) = H(-\infty) = 0$ and $F(\infty) = H(\infty)$, 
applying Proposition 4.1 of \cite{XGL19b} we get that 
\begin{align}\label{BerryEsseen001}
  \sup_{t \in \mathbb{R}}  \big| F(t) - H(t)  \big|
\leq  \frac{1}{\pi }  ( I_1 + I_2 + I_3 + I_4),
\end{align}
where
\begin{align*}
I_1  & =   o \Big( \frac{ 1 }{\sqrt{n}} \Big) \sup_{t \in \mathbb{R}} |H'(t)|,  
\quad    I_2   \leq C e^{-cn} \|\varphi \|_{\gamma},  
\quad  I_3    \leq  \frac{c}{n} \|\varphi \|_{\gamma},   \quad 
I_4   \leq  \frac{c}{n} \|\varphi \|_{\gamma}. 
\end{align*}
Here the bounds for $I_2$, $I_3$ and $I_4$ are obtained in \cite{XGL19b}. 
It is easy to see that 
\begin{align*}
I_1 =   o \Big( \frac{ 1 }{\sqrt{n}} \Big)  \mathbb{E}_{\mathbb{Q}_s^x} \Big[  \varphi(G_n \!\cdot\! x)  \Big]. 
\end{align*}
This, together with the fact that 
\begin{align*}
\mathbb{E}_{\mathbb{Q}_s^x} \Big[  \varphi(G_n \!\cdot\! x) \Big] \leq  \pi_s(\varphi) +  C e^{-cn} \|\varphi \|_{\gamma}
\end{align*}
(cf.\  \cite{XGL19b}), 
proves the theorem. 
\end{proof}

In the following we shall construct a partition $(\chi_{n,k}^y)_{k \geq 0}$ of the unity on the projective space $\bb P(V)$,
which is similar to the partitions in \cite{XGL19d, GQX20, DKW21}.  
In contrast to \cite{XGL19d, GQX20}, there is no escape of mass in our partition, 
which simplifies the proofs. 
Our partition becomes finer when $n \to \infty$, which allows us to obtain precise expressions for remainder terms
in the central limit theorem 
and thereby to establish the Edgeworth expansion for the coefficients. 

Let $U$ be the uniform distribution function on the interval $[0,1]$: $U(t)=t$ 
for $t\in [0,1]$, $U(t)=0$ for $t <0$ 
and $U(t)=1$ for $t > 1$. 
 Let $a_n=\frac{1}{\log n}$.
 Here and below we assume that $n \geq 18$ so that $a_n e^{a_n} \leq \frac{1}{2}$.
For any integer $k\geq 0$, define 
\begin{align*} 
U_{n,k}(t)= U\left(\frac{t-(k-1) a_n}{a_n}\right),  \qquad 
h_{n,k}(t)=U_{n,k}(t) - U_{n,k+1}(t),  \quad  t \in \bb R. 
\end{align*}
It is easy to see that $U_{n,m} = \sum_{k=m}^\infty h_{n,k}$ for any $m\geq 0$. 
Therefore, for any $t\geq 0$ and $m\geq0$, we have
\begin{align} \label{unity decomposition h-001}
\sum_{k=0}^{\infty} h_{n,k} (t) =1, \quad \sum_{k=0}^{m} h_{n,k} (t) + U_{n,m+1} (t) =1.
\end{align}
Note that for any $k\geq 0$, 
\begin{align} \label{h_kLip001}
\sup_{s,t\geq 0: s \neq t} \frac{ | h_{n,k}(s) - h_{n,k}(t) |}{|s-t|} \leq \frac{1}{a_{n}}. 
\end{align}
For any $x \in \bb P(V)$ and $y \in \bb P(V^*)$, set 
\begin{align}\label{Def-chi-nk}
\chi_{n,k}^y(x)=h_{n,k}(-\log \delta(x, y))  \quad  \mbox{and}  \quad 
\overline \chi_{n,k}^y(x)= U_{n,k} ( -\log \delta(x, y) ), 
\end{align}
where we recall that $-\log\delta(x, y) \geq 0$ for any $x\in \bb P (V)$ and $y \in \bb P(V^*)$. 
From \eqref{unity decomposition h-001}
 we have the following partition of the unity  on $\bb P(V)$: for any $x\in \bb P (V)$, $y \in \bb P(V^*)$ and $m\geq 0$,
\begin{align} \label{Unit-partition001}
\sum_{k=0}^{\infty} \chi_{n,k}^y (x) =1, \quad 
\sum_{k=0}^{m} \chi_{n,k}^y (x) + \overline \chi_{n,m+1}^y (x) =1.
\end{align}
Denote by $\supp (\chi_{n,k}^y)$ the support of the function $\chi_{n,k}^y$.  
It is easy to see that for any $k\geq 0$ and $y\in \bb P(V^*)$,
\begin{align} \label{on the support on chi_k-001}
 -\log \delta(x, y) \in [a_n (k-1), a_n(k+1)] \quad \mbox{for any}\ x\in \supp (\chi_{n,k}^y). 
\end{align}

\begin{lemma} \label{lemmaHolder property001}
There exists a constant $c>0$ such that 
for any $\gamma\in(0,1]$, 
 $k\geq 0$ and $y\in \bb P(V^*)$, it holds
  $\chi_{n,k}^y\in \scr B_{\gamma}$ and, moreover, 
\begin{align} \label{Holder prop ohCHI_k-001}
\| \chi_{n,k}^y \|_{\gamma} \leq \frac{c e^{\gamma k a_n}}{a_{n}^\gamma}.
\end{align}
\end{lemma}

\begin{proof} 
Since $\|\chi_{n,k}^y\|_{\infty} \leq 1$,
it is enough to give a bound for the modulus of continuity: 
\begin{align*} 
[\chi_{n,k}^y]_{\gamma} = \sup_{x',x''\in \bb P(V): x' \neq x''}\frac{|\chi_{n,k}^y(x') - \chi_{n,k}^y(x'')|}{d(x',x'')^{\gamma}},
\end{align*}
where $d$ is the angular distance on $\bb P(V)$ defined by \eqref{Angular-distance}. 
Assume that $x'=\bb R v'\in \bb P(V)$ and $x''=\bb R v''\in \bb P(V)$ are such that $\|v'\|=\|v''\|=1$. 
We note that 
\begin{align} \label{angular dist-bound001}
\| v'-v''\| \leq \sqrt{2}d(x',x''). 
\end{align}
For short, denote $B_k=((k-1)a_n,ka_n]$. Note that the function $h_{n,k}$ is increasing on $B_k$ and 
decreasing on $B_{k+1}$. 
Set for brevity
$t'=-\log \delta(x', y)$ and $t''=-\log \delta(x'', y)$. 
First we consider the case when $t'$ and $t''$ are such that $t',t''\in B_{k}$. 
Using \eqref{Def-chi-nk}, \eqref{h_kLip001}
and the fact that $|h_{n,k}| \leq 1$,  we have that for any $\gamma \in (0,1]$, 
\begin{align} \label{bounddCHI-001}
&|\chi_{n,k}^y(x')- \chi_{n,k}^y(x'')| = |h_{n,k}(t')- h_{n,k}(t'')|^{1-\gamma}  |h_{n,k}(t')- h_{n,k}(t'')|^{\gamma} \notag\\ 
&\leq 2 |h_{n,k}(t')- h_{n,k}(t'')|^{\gamma} \leq 2 \frac{|t'-t''|^{\gamma}}{a_{n}^{\gamma}} 
= \frac{2}{a_{n}^{\gamma}} |\log u'-\log u''|^{\gamma},
\end{align}
where we set for brevity $u'=\delta(x', y)$ and $u''= \delta(x'', y)$. 
Since $u'=e^{-t'}$, $u''=e^{-t''}$ and $t,t'\in B_{k}$, we have
$u''\geq e^{-k a_n}$ and $| u' -u''| \leq e^{-(k-1) a_n }-e^{-ka_n}$. 
Therefore, for $n \geq 18$, 
\begin{align*} 
\Big| \frac{u'}{u''}-1 \Big|= \Big| \frac{u'-u''}{u''} \Big| 
\leq \frac{ e^{-(k-1)a_n} -e^{-ka_n} }{e^{-ka_n}}
=e^{a_n}-1\leq a_n e^{a_n} 
\leq \frac{1}{2}, 
\end{align*}
which, together with the inequality $|\log(1+a)| \leq  2 |a|$ for any $|a| \leq \frac{1}{2}$, implies
\begin{align} \label{h_kLip002}
|\log u' - \log u''|= \left| \log\left(1+ \frac{u' - u''}{u''} \right) \right| \leq 2 \frac{|u'-u''|}{u''}.
\end{align}
Since $u''\geq e^{-k a_n}$, using the fact that $\|v'\|=\|v''\|=1$ and \eqref{angular dist-bound001}, we get
\begin{align} \label{Lipschitz-u}
\frac{|u'-u''|}{u''}
&\leq  e^{k a_n}  |\delta(x', y)- \delta(x'', y)| = e^{k a_n} \frac{|f(v') - f(v'') |}{\|f\|} \notag\\
& \leq  e^{k a_n} \| v'-v'' \| \leq \sqrt{2} e^{ ka_n} d(x',x'').
\end{align}
Therefore, from \eqref{bounddCHI-001}, \eqref{h_kLip002} and \eqref{Lipschitz-u}, it follows that for $\gamma \in (0,1]$, 
\begin{align} \label{bounddCHI-010}
|\chi_{n,k}^y(x')- \chi_{n,k}^y(x'')|  \leq  6 \frac{e^{\gamma k a_n}}{a_{n}^\gamma}d(x',x'')^\gamma.
\end{align}
The case $t',t''\in B_k$ is treated in the same way.

To conclude the proof we shall consider the case when
$t'=-\log\delta(x', y)\in B_{k-1}$ and $t''=-\log\delta(x'', y)\in B_{k}$; 
the other cases can be handled in the same way.
We shall reduce this case to the previous ones.
Let $x^*\in \bb P(V)$ be the point on the geodesic line $[x',x'']$ on $\bb P(V)$ 
such that $d(x',x'')=d(x',x^*)+d(x^*,x'')$ and $t^*=-\log\delta(y,x^*)=ka_n.$
Then 
\begin{align} \label{bounddCHI-011}
|\chi_{n,k}^y(x')- \chi_{n,k}^y(x'')| 
&\leq |\chi_{n,k}^y(x')- \chi_{n,k}^y(x^*)| + |\chi_{n,k}^y(x'')- \chi_{n,k}^y(x^*)| \notag\\
&\leq 
6 \frac{e^{\gamma k a_n}}{ a_{n}^\gamma }d(x',x^*)^{\gamma} 
+ 6 \frac{e^{\gamma k a_n}}{ a_{n}^\gamma }d(x'',x^*)^{\gamma} \notag\\
&\leq 
12 \frac{e^{\gamma k a_n}}{ a_{n}^\gamma }d(x',x'')^{\gamma}. 
\end{align}
From \eqref{bounddCHI-010} and \eqref{bounddCHI-011} we conclude that 
$[\chi_{n,k}^y]_{\gamma}\leq  12 \frac{e^{ \gamma k a_n }}{a_{n}^{\gamma}}$, which
shows \eqref{Holder prop ohCHI_k-001}.
\end{proof}

We need the following bounds.
Let $M_n=\floor{A\log^2 n}$, where $A>0$ is a constant and $n$ is large enough. 
For any measurable function $\varphi$ on $\bb P(V)$, it is convenient to denote 
\begin{align} \label{varphi-nk-001}
\varphi_{n,k}^y=\varphi \chi_{n,k}^y\quad\mbox{for}\quad  0 \leq k \leq M_n-1,\quad 
\varphi_{n,M_n}^y=\varphi \overline \chi_{n,M_n}^y.
\end{align}

\begin{lemma} \label{new bound for delta020} 
Assume \ref{Ch7Condi-Moment} and \ref{Ch7Condi-IP}. 
Then there exist constants $s_0 >0$ and $c >0$ such that for any $s \in (-s_0, s_0)$, 
$y\in \bb P(V^*)$ and any non-negative bounded measurable function $\varphi$ on $\bb P(V)$, 
\begin{align*} 
 \sum_{k=0}^{M_n} (k+1) a_n \pi_s(\varphi_{n,k}^y) \leq  -d_{s,\varphi}(y) + 2 a_n \pi_s(\varphi)
\end{align*}
and
\begin{align*} 
\sum_{k=0}^{M_n} (k-1) a_n \pi_s(\varphi_{n,k}^y) \geq -d_{s,\varphi}(y) - 2 a_n \pi_s(\varphi) -  c \frac{ \|\varphi \|_{\infty}}{n^2}.
\end{align*}
\end{lemma}
\begin{proof}
Recall that $d_{s,\varphi}(y)$ is defined in \eqref{drift-d001bis}.
Using \eqref{varphi-nk-001} and \eqref{Unit-partition001} we deduce that
\begin{align*} 
-d_{s,\varphi}(y) 
& = - \sum_{k=0}^{M_n} \int_{\bb P (V)} \varphi_{n,k}^y(x) \log\delta(x, y) \pi_s(dx) \notag \\
&\geq \sum_{k=0}^{M_n} (k-1) a_n \pi_s(\varphi_{n,k}^y)   
 = \sum_{k=0}^{M_n} (k+1) a_n \pi_s(\varphi_{n,k}^y) - 2 a_n \pi_s(\varphi),
\end{align*}
which proves the first assertion of the lemma.

Using the Markov inequality and the H\"older regularity of the invariant measure $\pi_s$ (Lemma \ref{Lem_Regu_pi_s00}), 
 we get that there exists a small constant $\eta >0$ such that 
\begin{align*} 
&- \int_{\bb P(V)} \varphi_{n,M_n}^y(x)\log \delta(x, y) \pi_s(dx)  \notag\\
& \leq c \|\varphi \|_{\infty} \int_{\bb P(V)} \frac{e^{-\eta A\log n}}{\delta(x, y)^{\eta}} \delta(x, y)^{-\eta} \pi_s(dx) \notag \\
& = c \|\varphi \|_{\infty} e^{-\eta A \log n}  \int_{\bb P(V)}  \delta(x, y)^{-2\eta} \pi_s(dx) 
\leq  c \frac{ \|\varphi \|_{\infty}}{ n^2 },
\end{align*}
where in the last inequality we choose $A >0$ to be sufficiently large so that $\eta A \geq 2$. 
Therefore, 
\begin{align*} 
-d_{s,\varphi}(y) 
&= - \sum_{k=0}^{M_n} \int_{\bb P (V)} \varphi_{n,k}^y(x) \log\delta(x, y) \pi_s(dx) \notag \\
&\leq \sum_{k=0}^{M_n-1} (k+1) a_n \pi_s(\varphi_{n,k}^y) +  c \frac{ \|\varphi \|_{\infty}}{ n^2 } \notag  \\
& \leq \sum_{k=0}^{M_n-1} (k-1) a_n \pi_s(\varphi_{n,k}^y) + 2 a_n \pi_s(\varphi) +  c \frac{ \|\varphi \|_{\infty}}{ n^2 }  \notag\\
& \leq \sum_{k=0}^{M_n} (k-1) a_n \pi_s(\varphi_{n,k}^y) + 2 a_n \pi_s(\varphi) +  c \frac{ \|\varphi \|_{\infty}}{ n^2 }. 
\end{align*}
This proves the second assertion of the lemma.
\end{proof}

\begin{proof}[Proof of Theorem \ref{Thm-Edge-Expan-Coeff001extended}]
Without loss of generality, we assume that the target function $\varphi$ is non-negative. 
With the notation in \eqref{varphi-nk-001}, we have that for $t \in \bb R$, 
\begin{align}\label{Initial decompos-001aa}
I_n(t) &: =\bb{E}_{\mathbb{Q}_s^x} \left[ \varphi(G_n \!\cdot\! x) 
\mathds{1}_{ \big\{ \frac{\log |\langle f, G_n v \rangle| - n\Lambda'(s) }{ \sigma_s \sqrt{n} } \leq t  \big\}  } \right] \notag \\ 
& =  \sum_{k=0}^{M_n} 
\bb{E}_{\mathbb{Q}_s^x} \left[ \varphi_{n,k}^y (G_n \!\cdot\! x) 
\mathds{1}_{ \big\{ \frac{\log |\langle f, G_n v \rangle| - n\Lambda'(s) }{ \sigma_s \sqrt{n} } \leq t  \big\}  }\right] 
=: \sum_{k=0}^{M_n}  F_{n,k}(t).
\end{align}
For $0\leq k\leq M_n - 1$, using \eqref{Ch7_Intro_Decom0a} 
and the fact that $-\log \delta(x, y) \leq (k+1)a_n$ when $x \in \supp \varphi_{n,k}^y$, we get
\begin{align} \label{boundFfbar-001}
F_{n,k}(t)
\leq
\bb{E}_{\mathbb{Q}_s^x} \left[ \varphi_{n,k}^y (G_n \!\cdot\! x) 
\mathds{1}_{ \big\{ \frac{\sigma(G_n, x) - n\Lambda'(s) }{ \sigma_s \sqrt{n} } \leq t + \frac{(k+1) a_n}{\sigma_s \sqrt{n}} \big\}  }   \right] 
=:  H_{n,k}(t).
\end{align}
For $k=M_n$, we have
\begin{align} \label{boundFfbar-002}
F_{n,M_n}(t)
 &\leq 
 \bb{E}_{\mathbb{Q}_s^x} \left[ \varphi_{n,M_n}^y (G_n \!\cdot\! x) 
\mathds{1}_{ \big\{ \frac{ \sigma(G_n,x) - n\Lambda'(s) }{ \sigma_s \sqrt{n} } \leq t +\frac{(M_n+1) a_n}{\sigma_s \sqrt{n}}  \big\}  }\right] 
\notag \\
&\quad +\bb{E}_{\mathbb{Q}_s^x} \left[ \varphi_{n,M_n}^y (G_n \!\cdot\! x) 
\mathds{1}_{ \big\{ -\log \delta(G_n \cdot x, y)  \geq (M_n +1)a_n  \big\}  }\right] \notag \\
&=: H_{n,M_n}(t) +  W_{n}.  
\end{align}
By Lemma \ref{Lem_Regu_pi_s} and choosing $A >0$ large enough, we get
\begin{align} \label{W_n001}
W_{n} &\leq  \|\varphi \|_{\infty}  \mathbb{Q}_s^x ( -\log \delta(G_n \cdot x, y)  \geq A \log n )  \notag\\
&\leq   \frac{c_0}{n^{c_1 A}}   \|\varphi \|_{\infty}
\leq  \frac{c_0}{n^{2}}  \|\varphi \|_{\infty}.
\end{align}
Now we deal with $H_{n,k}(t)$ for $0 \leq k \leq M_n$.  
Denote for short $t_{n,k} = t +\frac{(k+1) a_n}{\sigma_s \sqrt{n}}$. 
Applying the Edgeworth expansion (Theorem \ref{Thm-Edge-Expan}) we obtain
that, uniformly in $s \in (-s_0, s_0)$, $x \in \bb P(V)$, 
$t \in \bb R$, $0 \leq k \leq M_n$ and 
$\varphi \in \mathscr{B}_{\gamma}$, as $n \to \infty$,
\begin{align*}
  H_{n,k}(t)
 &   =   \pi_s(\varphi_{n,k}^y) \Big[  \Phi(t_{n,k}) + \frac{\Lambda'''(s)}{ 6 \sigma_s^3 \sqrt{n}} (1-t_{n,k}^2) \phi(t_{n,k}) \Big]
 \notag\\ 
 & \quad - \frac{ b_{s, \varphi_{n,k}^y}(x) }{ \sigma_s \sqrt{n} } \phi(t_{n,k})    
     +  \pi_s(\varphi_{n,k}^y)  o \Big( \frac{ 1 }{\sqrt{n}} \Big)  
  +  \lVert \varphi_{n,k}^y \rVert_{\gamma} O \Big( \frac{ 1 }{n} \Big).   
\end{align*}
Recall that $a_n=\frac{1}{\log n}$ and $M_n=\floor{A\log^2 n}$.  
By the Taylor expansion we have, 
uniformly in $s \in (-s_0, s_0)$, $x \in \bb P(V)$,  $t \in \bb R$ and $0 \leq k \leq M_n$, 
\begin{align*} 
\Phi(t_{n,k}) 
& =\Phi(t) + \phi(t) \frac{(k+1) a_n}{\sigma_s \sqrt{n}} + O\Big(\frac{\log^2 n}{n} \Big)
\end{align*}
and 
\begin{align*} 
(1-t_{n,k}^2) \phi(t_{n,k}) 
&= (1-t^2) \phi(t) + O\left( \frac{\log n }{\sqrt{n}} \right).
\end{align*}
Moreover, using Lemma \ref{Lem-Bs}, we see that
\begin{align*} 
\frac{ b_{s, \varphi_{n,k}^y}(x) }{ \sigma_s \sqrt{n} } \phi(t_{n,k}) 
= \frac{ b_{s, \varphi_{n,k}^y}(x) }{ \sigma_s \sqrt{n} } \phi(t) 
+ \| \varphi_{n,k}^y \|_{\gamma} O \Big(\frac{\log n}{n}\Big). 
\end{align*}
Using these expansions and \eqref{boundFfbar-001}, \eqref{boundFfbar-002} and \eqref{W_n001}, we get
that there exists a sequence $(\beta_n)_{n \geq 1}$ of positive numbers
satisfying $\beta_n \to 0$ as $n \to \infty$, such that for any $0 \leq k \leq M_n$, 
\begin{align}\label{F_nkbound001}
   F_{n,k}(t)
 &   \leq   \pi_s(\varphi_{n,k}^y) \Big[  \Phi(t) + \frac{\Lambda'''(s)}{ 6 \sigma_s^3 \sqrt{n}} (1-t^2) \phi(t) \Big]\notag\\
  &\quad - \frac{ b_{s, \varphi_{n,k}^y}(x) }{ \sigma_s \sqrt{n} } \phi(t) 
  + \frac{\phi(t)}{\sigma_s \sqrt{n}} \pi_s(\varphi_{n,k}^y)(k+1) a_n   \notag\\
  & \quad  +  \pi_s(\varphi_{n,k}^y)  \frac{ \beta_n }{\sqrt{n}}  
  +   \lVert \varphi_{n,k}^y \rVert_{\gamma}  \frac{ c \log n }{n}.   
\end{align}
By Lemma \ref{lemmaHolder property001}, it holds that for any $\gamma \in (0, 1]$ and $0 \leq k \leq M_n$, 
\begin{align}\label{HolderNorm-varphik}
\lVert \varphi_{n,k}^y \rVert_{\gamma} 
\leq  c \lVert \varphi \rVert_{\infty} n^{\gamma A} \log^{\gamma} n  +  \lVert \varphi \rVert_{\gamma}. 
\end{align}
From \eqref{Def-bsvarphi}, it follows that $b_{s,\varphi}(x) = \sum_{k= 0}^{M_n} b_{s, \varphi_{n,k}^y}(x)$. 
Therefore, summing up over $k$ in \eqref{F_nkbound001},  using \eqref{HolderNorm-varphik} 
and taking $\gamma >0$ to be sufficiently small such that $\gamma A < \ee/2$, we obtain
\begin{align*}
 I_n(t) = \sum_{k= 0}^{M_n}  F_{n,k}(t) & \leq    \pi_s(\varphi) 
   \Big[  \Phi(t) + \frac{\Lambda'''(s)}{ 6 \sigma_s^3 \sqrt{n}} (1-t^2) \phi(t) \Big]  \notag\\
& \quad -  \frac{ b_{s,\varphi}(x) }{ \sigma_s \sqrt{n} } \phi(t)
+ \frac{\phi(t)}{\sigma_s \sqrt{n}} \sum_{k = 0}^{M_n} \pi_s(\varphi_{n,k}^y)(k+1)  a_n \notag\\
&  \quad + \pi_s(\varphi)   \frac{ \beta_n }{\sqrt{n}}  +  
    \lVert \varphi \rVert_{\gamma}  \frac{ c }{n^{1 - \ee}}. 
\end{align*}
Using Lemma \ref{new bound for delta020} and the fact that $a_n \to 0$ as $n \to \infty$, 
we obtain the desired upper bound.

The lower bound is established in the same way. 
Instead of \eqref{boundFfbar-001} we use the following lower bound, which is obtained
using \eqref{Ch7_Intro_Decom0a} 
and the fact that $-\log \delta(x, y) \geq (k-1)a_n$ for $x \in \supp \varphi_{n,k}^y$ and $0\leq k\leq M_n$, 
\begin{align} \label{boundFfbar-003}
F_{n,k}(t)
\geq
\bb{E}_{\mathbb{Q}_s^x} \left[ \varphi_{n,k}^y (G_n \!\cdot\! x) 
\mathds{1}_{ \big\{ \frac{\sigma(G_n, x) - n\Lambda'(s) }{ \sigma_s \sqrt{n} } \leq t +\frac{(k-1) a_n}{\sigma_s\sqrt{n}} \big\}  }   \right]. 
\end{align}
Proceeding in the same way as in the proof of the upper bound, 
 using \eqref{boundFfbar-003}  
instead of \eqref{boundFfbar-001} and \eqref{boundFfbar-002},  we get
\begin{align*} 
 I_n(t) = \sum_{k= 0}^{M_n}  F_{n,k}(t)  & \geq  \pi_s(\varphi) 
   \Big[  \Phi(t) + \frac{\Lambda'''(s)}{ 6 \sigma_s^3 \sqrt{n}} (1-t^2) \phi(t) \Big]  \notag\\
&\quad -    \frac{ b_{s, \varphi}(x) }{ \sigma_s \sqrt{n} } \phi(t)
+ \frac{\phi(t)}{\sigma_s \sqrt{n}} \sum_{k = 0}^{M_n}  \pi_s(\varphi_{n,k}^y)(k-1) a_n \notag\\
&  \quad + \pi_s(\varphi)  o \Big( \frac{ 1 }{\sqrt{n}} \Big)  +  
    \lVert \varphi \rVert_{\gamma} O \Big( \frac{1 }{n^{1 - \ee}} \Big). 
\end{align*}
The lower bound is obtained using again Lemma \ref{new bound for delta020}
and the fact that $a_n \to 0$ as $n \to \infty$. 
\end{proof}


\end{document}